\documentclass[11pt,reqno]{amsart}
\usepackage{amssymb,latexsym}
\usepackage{graphicx}
\usepackage{fancyhdr,amssymb}
\usepackage{url}		%does nice formatting of URLs

\newcommand{\N}{{\mathbb N}}

\theoremstyle{plain}

\newtheorem{theo}{Theorem}[section]

\newtheorem{lemma}[theo]{Lemma}

\newtheorem*{zeck}{Zeckendorf's theorem}

\theoremstyle{definition}

\newtheorem*{rem}{Remark}

\begin{document}
\fancyhead{}
\renewcommand{\headrulewidth}{0pt}
\fancyfoot{}
\fancyfoot[LE,RO]{\medskip \thepage}
\fancyfoot[LO]{\medskip MONTH YEAR}
\fancyfoot[RE]{\medskip VOLUME , NUMBER }

\setcounter{page}{1}

\newcommand\numberthis{\addtocounter{equation}{1}\tag{\theequation}}
\newcommand{\fib}{\mathrm{fib}}
\newcommand{\of}{\mathrm{odfib}}

\title[Odd fibbinary numbers]{Odd fibbinary numbers\\ and the golden ratio}

\author{Linus Lindroos}
\address{Linus Lindroos\\
Department of Mathematical Sciences\\
        Georgia Southern University \\
Statesboro, GA 30460, USA}
\email{ll01383@georgiasouthern.edu}

\author{Andrew Sills}
\address{Andrew Sills\\
        Department of Mathematical Sciences\\
        Georgia Southern University \\
Statesboro, GA 30460, USA
}
\email{asills@georgiasouthern.edu}

\author{Hua Wang}
\address{Hua Wang\\
Department of Mathematical Sciences \\
Georgia Southern University \\
Statesboro, GA 30460, USA
}
\email{hwang@georgiasouthern.edu}
\thanks{The work of the third author was partially supported by a grant from the Simons Foundation (\#245307).}

\begin{abstract}
The fibbinary numbers are positive integers whose binary representation contains no consecutive ones.
We prove the following result:  If the $j$th odd fibbinary is the $n$th \emph{odd} fibbinary number, then $j = \lfloor n\phi^2 \rfloor - 1.$
\end{abstract}

\maketitle

\section{Background}
The Fibonacci numbers $\{F_n\}_{n\geq 0}$ are given by $F_0=0$, $F_1=1$, and
$$F_n = F_{n-1} + F_{n-2}$$ for $n\geq 2.$

Recall the following theorem of Zeckendorf~\cite{ez}:
\begin{zeck} Every positive integer can be written uniquely as the
sum of distinct, nonconsecutive Fibonacci numbers.  
\end{zeck}

The \emph{Zeckendorf representation} $z(n)$ of $n\in\N$ is the unique $k$-tuple
of decreasing nonconsecutive Fibonacci numbers whose sum is $n$.
Note that although $F_2 = F_1 = 1$, we always associate $1$ with $F_2$ in the Zeckendorf
representation.

  For example,
$$ z(4) = (3,1) = (F_4, F_2), $$
$$ z(5) = (5) = (F_5), $$
and
$$ z(100) = (89,8,3) = (F_{11} , F_6 , F_{4}).$$

The sequence $\{ \fib(n) \}_{n\geq 1}$ of \emph{fibbinary numbers}\footnote{According to~\cite{oeis}, the name ``fibbinary'' is due to Marc LeBrun: ``\dots integers whose 
binary representation contains no consecutive ones and noticed that the number of
such numbers with $n$ bits was $F_n$."---posting to \texttt{sci.math} by Bob Jenkins, July
17, 2002. } is given as follows:
For $n>0$, if $z(n) = (F_{i_1}, F_{i_2}, \dots F_{i_k})$ is the Zeckendorf representation of $n$,  then
$$\fib(n) := \sum_{j=1}^k 2^{i_j-2}.$$
%We also define $\fib(0) := 0$.

For example,
$$ \fib(4) = \fib(F_4 + F_2) = 2^{4-2} + 2^{2-2} = 101_2 = 5. $$
$$ \fib(5) = \fib(F_5) = 2^{5-3} = 1000_2 = 8. $$
$$ \fib(100) = \fib({F_{11}+F_6+F_4})
 = 2^{11-2} + 2^{6-2} + 2^{4-2} = 1000010100_2 = 532,$$
where the $2$ subscript indicates the usual binary (base two) representation.

\section{Statement of main result relating odd fibbinaries to the golden ratio}

It is easy to generate the odd fibbinary numbers from the binary representations and find the corresponding value of the original integer, e.g.,
\begin{itemize}
\item $1_2 = 2^0 = \fib( F_2) = \fib(1)$;
\item $101_2 = 2^2 + 2^0 =\fib( F_4 + F_2) = \fib(4)$;
\item $1001_2 = 2^3 + 2^0 = \fib(F_5 + F_2) = \fib(6)$;
\item $\ldots$.
\end{itemize}

Let $\of(n)$ denote the $n$th odd fibbinary number, i.e.,
$$ \of(1) = 1_2, \of(2) =101_2, \of(3) = 1001_2, \ldots .$$
Let  $$Z(n):= \fib^{-1}(\of(n)),$$ so that $$ Z(1) =1, Z(2) = 4, Z(3) = 6, Z(4)=9,\ldots . $$
In other words, if the $n$th odd fibbinary number is the $j$th fibbinary number, then $j=Z(n)$.

This sequence
$$ 1,4,6,9,12,14,17, \ldots $$
appears to be ``A003622'' in OEIS~\cite{oeis2}, defined as
\begin{equation}\label{eq:phi}
 \{ \left\lfloor n \phi^2 \right\rfloor -1 \}_{n=1}^{\infty} = \{ \left\lfloor n \phi \right\rfloor +n -1 \}_{n=1}^{\infty}
\end{equation}
where $$\phi = \frac{1 + \sqrt{5}}{2}$$ is the golden ratio,  which satisfies $\phi^2 = \phi + 1$, and of course arises as
\[ \lim_{n\to\infty} \frac{F_{n+1}}{F_n} = \phi . \]

\begin{tabular}{| l |  r |  r | c |}
\hline\hline
$j$ & $z(j)$ & $\fib(j)$ & $\of^{-1}(j) $\\
\hline\hline
$1 = \lfloor \phi^2 \rfloor - 1 = Z(1)$ &  $(1) = (F_2)$  & $1_2 = 1$ & 1\\ 
\hline
$2$                                             &  $(2) = (F_3)$ & $10_2=2$ & ---\\
$3$  &  $(3) = (F_4)$              &  $100_2 = 4 $ & --- \\
$4$  = $\lfloor 2 \phi^2 \rfloor -1=Z(2) $ & $(3,1) = (F_4, F_2)$ & $101_2 = 5$ & $2$\\
\hline
$5$ & $(5) = (F_5)$  & $1000_2 = 8$ & --- \\
$6 = \lfloor 3 \phi^2 \rfloor -1 =Z(3)$ & $(5,1)= (F_5,F_2)$ & $1001_2 = 9$ & $3$ \\
\hline
$7$ & $(5,2) = (F_5, F_3)$ & $1010_2=10$ &--- \\
$8$ & $(8) = (F_6) $ & $10000_2 = 16$&  --- \\
$9 = \lfloor 4 \phi^2 \rfloor -1 =Z(4)$ & $(8,1) = (F_6,F_2)$ & $10001_2 = 17$ & $4$ \\
\hline
$10$ & $(8,2)= (F_6, F_3)$ & $10010_2=18$ & --- \\
$11$ & $(8,3) = (F_6,F_4)$ & $10100_2=20$ &---\\
$12=\lfloor 5 \phi^2 \rfloor -1=Z(5)$ & $(8,3,1) = (F_6,F_4,F_2)$ & $10101_2 = 21$ & $5$ \\
\hline
\end{tabular}
\vskip 3mm
The correspondence displayed above is naturally conjectured to be true in general.

\begin{theo}\label{theo:main}
   Let $j$ be a positive integer such that the $j$th fibbinary number is odd.  Suppose that this $j$th fibbinary number is the $n$th \emph{odd} fibbinary number.  Then 
\begin{equation}\label{eq:id}  
j= Z(n) = \left\lfloor n \phi^2 \right\rfloor -1 = \left\lfloor n \phi \right\rfloor +n-1
 \end{equation}
for any $n\geq 1$.
\end{theo}

\section{Proof of Theorem~\ref{theo:main}}
It is easy to check that \eqref{eq:id} is true for small values of $n$. In the rest of this note we provide a general proof for $n\geq 3$.

Next, we record the following observation on $\of(n)$ as a lemma.

\begin{lemma}\label{prop:fib}
For any $1 \leq k \leq F_{n-1}$ and $n\geq 2$, we have
$$ \of(F_n + k) = 2^{n} + \of(k) . $$
\end{lemma}
\begin{proof} Immediate as each digit in the binary representation of $\of(n)$ corresponds to a specific Fibonacci number. \end{proof}

For example, 
$$ \of(10)=\of(8+2)=\of(F_6 + 2)=1000101_2 = 2^6 + 101_2 = 2^6 + \of(2) . $$
\vskip 5mm

\begin{lemma}\label{prop:z}
For any $1 \leq k \leq F_{n-1}$ and $n\geq 2$, we have
$$ Z(F_n + k) = F_{n+2} + Z(k) . $$
%where we define $Z(0):=0$.
\end{lemma}
\begin{proof} Since each digit in the binary representation of $\of(n)$ corresponds to a distinct Fibonacci number in the sum of $Z(n)$ under the Zeckendorf representation, we can claim the same for $Z(n)$. That is,
\begin{align*}
Z(F_n + k) & =  \fib^{-1} (\of (F_n + k) ) \\
   & = \fib^{-1} (2^n + \of (k) ) \\
   & = F_{n+2} + \fib^{-1} (\of (k)) \\
   & = F_{n+2} + Z(k) .
\end{align*}
\end{proof}

For example,
$$ Z(10) = Z(F_6 + 2) = 25 = 21 + 4 = F_8 + Z(2). $$

From Lemma~\ref{prop:z}, we have
\begin{equation}\label{eq:z_d}
Z(F_n + k) - Z(F_n + k -1) = Z(k) - Z(k-1)
\end{equation}
for $2\leq k \leq F_{n-1}$ and
\begin{equation}\label{eq:zn}
Z(F_n + 1) = F_{n+2} + 1 .
\end{equation}

Now to show \eqref{eq:id} for any $n$, we only need to show analogues of \eqref{eq:z_d} and \eqref{eq:zn} for $ \left\lfloor n \phi^2 \right\rfloor -1 $, i.e., 
\begin{align*}
  \left\lfloor (F_n + k) \phi \right\rfloor - \left\lfloor (F_n + k - 1) \phi \right\rfloor 
&=  \left\lfloor (F_n + k) \phi^2 \right\rfloor - \left\lfloor (F_n + k - 1) \phi^2 \right\rfloor \\
&=  \left\lfloor k \phi^2 \right\rfloor - \left\lfloor (k - 1) \phi^2 \right\rfloor \\
&=  \left\lfloor k \phi \right\rfloor - \left\lfloor (k - 1) \phi \right\rfloor  \numberthis \label{eq:z_d'}
\end{align*}
for $2\leq k \leq F_{n-1}$, and
\begin{equation}\label{eq:zn'}
 \left\lfloor (F_n + 1) \phi \right\rfloor + F_n = \left\lfloor (F_n + 1) \phi^2 \right\rfloor -1 =  Z(F_n + 1) .
\end{equation}

\begin{rem}
Intuitively, this can be considered as using $Z(F_n + 1)$ as the ``stepping stone'' to prove \eqref{eq:id} for 
$Z(F_n+k)$ for $k = 2, 3, \ldots , F_{n-1}$ using induction.
\end{rem}

In order to establish \eqref{eq:z_d'}, it is essentially sufficient to show that $\{ F_n \phi \}$  is never large enough to affect the difference $ \left\lfloor k \phi \right\rfloor - \left\lfloor (k - 1) \phi \right\rfloor $,
where $\{ x \}:= x - \lfloor x \rfloor$, the fractional part of the real number $x$.

We make use of the following fact, which is easily established by induction in $n$:
\begin{equation}\label{eq:fact}
(-1)^n\tau^n = -F_n \tau + F_{n-1}, 
\end{equation}
where $\tau = \frac{\sqrt{5} - 1}{2} = \phi - 1$ satisfying $\tau^2 = -\tau + 1$.
Consequently,
$$ \{F_n \phi \} = \{ F_n \tau \} = \{ -(-\tau)^n \} $$
for any $n$.

Making use of the fact that $\tau = \phi - 1$, it suffices to show
\begin{equation}\label{eq:tau}
\left\lfloor (F_n + k) \tau \right\rfloor - \left\lfloor (F_n + k - 1) \tau \right\rfloor
= \left\lfloor k \tau \right\rfloor - \left\lfloor (k - 1) \tau \right\rfloor . 
\end{equation}

To show \eqref{eq:tau}, simply consider
$$ \{k \tau \} \pm \{ F_n \tau \} $$
for any $1\leq k\leq F_{n-1}$. We will show that this value never reaches 1 or goes below zero and hence $\{ F_n \tau \}$ will not affect 
$\left\lfloor k \tau \right\rfloor - \left\lfloor (k - 1) \tau \right\rfloor $. 

\noindent {\bf (i)} To show $ \{k \tau \} + \{ F_n \tau \} <1 $, consider the Zeckendorf representation of $k$ 
as the sum of non-consecutive Fibonacci numbers. 

If $k < F_{n-1}$, then
$$ k = F_{a_1} + F_{a_2} + \ldots + F_{a_s} $$
where 
$$ 1\leq a_1 \leq a_2 - 2 \leq a_2 \leq a_3 - 2 \leq \ldots \leq a_s \leq n-2 . $$
Then
\begin{align*}
   \{k \tau \} + \{ F_n \tau \} 
&=   \{ (F_{a_1} + \ldots +  F_{a_s} ) \tau \}  + \{ F_n \tau \} \\
&\leq  \tau^{a_1} + \tau^{a_2} + \ldots + \tau^{a_s} + \tau^n \\
&<  \tau + \tau^3 + \tau^5 + \ldots \\
&=   \frac{\tau}{1 - \tau^2}  \\
&=   1.
\end{align*}

If $k=F_{n-1}$, we have
$$ \{k \tau \} + \{ F_n \tau \}  =   \tau^{n-1}  + \tau^{n} \leq   \tau^{2} +  \tau^{3}  <  1 $$
for any $n\geq 3$.

\noindent {\bf (ii)} To show $ \{k \tau \} - \{ F_n \tau \} >0 $, simply note that 
\begin{align*}
   \{k \tau \} - \{ F_n \tau \} 
&=   \{ (F_{a_1} + \ldots +  F_{a_s} ) \tau \}  - \{ F_n \tau \}  \\
&\geq  \tau^{a_1} - \tau^{a_2} - \ldots - \tau^{a_s} - \tau^n \\
&>  \tau^{a_1} - \tau^{a_1+2} - \tau^{a_1+4} - \ldots  \\
& =  \tau^{a_1}\left( 1 - \frac{\tau^2}{1-\tau^2} \right) \\
& >   0
\end{align*}
if $k<F_{n-1}$ and
$$ \{k \tau \} - \{ F_n \tau \}  =  \tau^{n-1}  - \tau^{n}  >  0 $$
if $k=F_{n-1}$. 

Cases (i) and (ii) imply that
$$ \left\lfloor F_n \tau + k \tau \right\rfloor = \left\lfloor F_n \tau \right\rfloor + \left\lfloor k \tau \right\rfloor . $$Thus \eqref{eq:tau} and \eqref{eq:z_d'} are proved.

By \eqref{eq:zn}, \eqref{eq:zn'} is equivalent to
\begin{equation}\label{eq:zn''}
\left\lfloor (F_n + 1) \phi \right\rfloor = F_{n+1} + 1 .
\end{equation}
Fact \eqref{eq:fact} implies that
\begin{align*}
 \left\lfloor (F_n + 1) \phi \right\rfloor 
&=  F_n + 1 + \left\lfloor (F_n + 1) \tau \right\rfloor \\
&=  F_n + 1 + \left\lfloor F_{n-1} - (-\tau)^n + \tau \right\rfloor \\
&=  F_n + 1 + F_{n-1} + \left\lfloor \tau - (-\tau)^n  \right\rfloor \\
&=  F_{n+1} + 1
\end{align*} 
for $n\geq 3$. Thus \eqref{eq:zn''} and hence \eqref{eq:zn'} is proved. \qed

\section*{Acknowledgements}
We thank the anonymous referee for carefully reading the manuscript and providing helpful suggestions.

\medskip

\noindent MSC2010: 11B39

\end{document}